\documentclass[12pt,a4paper,oneside]{amsart}

\usepackage[a4paper]{geometry}
\usepackage{mathptmx}
\DeclareMathAlphabet{\mathcal}{OMS}{cmsy}{m}{n} 

\usepackage{amssymb,graphicx,perpage}

\newtheorem*{theorem}{Theorem}
\newtheorem{proposition}{Proposition}
\newtheorem{lemma}{Lemma}
\newtheorem*{conjecture}{Conjecture}

\MakePerPage[2]{footnote}

\def\liebrack  {\ensuremath{[\,\cdot\, , \cdot\,]}}
\def\tb{\vcenter{\hbox{\tiny$\>\bullet\>$}}}

\newcommand{\set}[2]{\ensuremath{\{ #1 \>|\> #2 \}}}

\DeclareMathOperator{\B}{B}
\DeclareMathOperator{\C}{C}
\DeclareMathOperator{\dcobound}{d}
\DeclareMathOperator{\Der}{Der}
\DeclareMathOperator{\gl}{\mathsf{gl}}
\DeclareMathOperator{\Homol}{H}
\DeclareMathOperator{\im}{Im}
\DeclareMathOperator{\Ker}{Ker}
\DeclareMathOperator{\Z}{Z}

\begin{document}

\title{On anticommutative algebras for which $[R_a,R_b]$ is a derivation}

\author{Ivan Kaygorodov}
\address{Federal University of ABC, Brazil}
\email{kaygorodov.ivan@gmail.com}

\author{Pasha Zusmanovich}
\address{University of Ostrava, Czech Republic}
\email{pasha.zusmanovich@gmail.com}

\date{First written December 1, 2020; last revised January 17, 2021}
\thanks{J. Geom. Phys., to appear}


\begin{abstract}
We study anticommutative algebras with the property that commutator of any
two multiplications is a derivation.
\end{abstract}

\maketitle

\section*{Introduction}

Consider the following property of a (nonassociative) algebra: the commutator of
two (say, right) multiplications is a derivation of an algebra. Commutative
algebras with this property were studied in the literature under the names
``Lie triple algebras'' and ``almost Jordan algebras'': see \cite{bertram}, 
\cite{dzhu-comm}, \cite{hp}, \cite{jord-ruhaak}, \cite{sidorov}, and references
therein. Jordan algebras are properly contained in this class.

It appears only natural to consider then the anticommutative analog: that is, 
anticommutative algebras in which the commutator of any two multiplications is
a derivation. Such algebras are dubbed, for no better term, as CD algebras. It 
turns out that the variety of CD algebras lies between Lie algebras and 
binary Lie algebras, both inclusions being strict, so it seems a class of 
algebras worth studying.

Earlier, low-dimensional nilpotent CD algebras were classified in 
\cite{cd-nilp}, \cite{cd-nilp1}, \cite{kk}, and \cite{kk2}, and here we continue
the study of CD algebras.

\section{Notation, conventions, preliminary facts}

We consider (nonassociative) algebras over the ground field $K$, which is 
assumed to be arbitrary of characteristic $\ne 2,3$. \emph{All algebras and 
varieties of algebras are assumed to be anticommutative, without explicitly 
mentioning it}. Thus, left and right multiplications differ only by sign, left 
and right ideals coincide, etc.
	
The multiplication in algebras will be always denoted by juxtaposition, with the
exception of \S \ref{sec-ext}, where we deal with Lie algebras with 
multiplication denoted customarily by brackets $\liebrack$.

Let $A$ be an algebra. For an element $a \in A$, $R_a: A \to A$ denotes the 
linear map of right multiplication on $a$: $R_a(x) = xa$. $\Der(A)$ denotes the
Lie algebra of derivations of $A$, and $\gl(A)$ denotes the Lie algebra of all 
linear maps $A \to A$, subject to the usual commutator 
$[f,g] = g \circ f - f \circ g$ (we always assume action from the right, so 
$g \circ f(a) = g(f(a))$). For any algebra $A$, any $a\in A$, and any 
$D \in \Der(A)$, the following identity in $\gl(A)$ holds: 
\begin{equation}\label{eq-rda}
[D,R_a] = R_{D(a)} .
\end{equation}

$Z(A)$ denotes the center of $A$, i.e., the set of elements $z \in A$ such that 
$zA = Az = 0$; this is obviously an ideal of $A$. 

The \emph{Jacobiator} $J(x,y,z)$ of elements $x,y,z \in A$ is defined as
$$
J(x,y,z) = (xy)z + (zx)y + (yz)x .
$$
Thus, an algebra is a Lie algebra if and only if the Jacobiator on it is 
identically zero.

An algebra $A$ is called a \emph{CD algebra} if it satisfies the property that 
for any $a,b\in A$, the commutator $[R_a, R_b]$ is a derivation of $A$. This 
condition can be written as a homogeneous identity of degree $4$ comprising $6$
monomials:
\begin{equation}\label{eq-6}
((xy)a)b - ((xy)b)a - ((xa)b)y + ((xb)a)y + ((ya)b)x - ((yb)a)x = 0 .
\end{equation}

\section{Identities and non-identities}\label{s-i}

Let us compare the variety of CD algebras with the other known varieties: Lie, 
binary Lie, Malcev, and Sagle (for binary Lie and Malcev algebras, see, for 
example, \cite{sagle-malcev}, \cite{sagle-binlie}, \cite{grishkov}, 
\cite{arenas-shest}, \cite{cd-nilp} and references therein, and for Sagle 
algebras, see \cite{filippov} and references therein). Let us briefly recall 
their definitions. An algebra is called \emph{binary Lie} if it satisfies the identity
$$
J(xy, x, y) = 0 .
$$
Taking into account anticommutativity, the latter identity is equivalent to
\begin{equation}\label{eq-binlie}
((xy)x)y = ((xy)y)x .
\end{equation}
This is also equivalent to the condition that any $2$-generated subalgebra is 
Lie.

An algebra is called \emph{Malcev}, if it satisfies the identity
$$
J(x,y,xz) = J(x,y,z)x ;
$$
and an algebra is called \emph{Sagle}, if it satisfies the identity
\begin{equation}\label{eq-sagle}
J(x,y,z)w = J(w,z,xy) + J(w,y,zx) + J(w,x,yz) .
\end{equation}

There are the following well known strict inclusions between these varieties:
$$
\begin{array}{rcl}
           && \text{Malcev} \>\subset\> \text{ Binary Lie} \\
           & \rotatebox{45}{$\subset$}  &                                   \\
\text{Lie} &                                                                \\
           & \rotatebox{-45}{$\subset$} &                                   \\
           &                            & \text{Sagle}
\end{array}
$$

How CD algebras fit into the picture? We are going to prove that
\begin{equation}\label{eq-inc}
\begin{array}{rcl}
&& \text{Binary Lie}                                  \\
& \rotatebox{45}{$\subset$}                           \\
\text{Lie} \>\subset\> \text{Malcev} \cap \text{Sagle} \>\subset\> \text{CD} \\
& \rotatebox{-45}{$\subset$}                          \\
&& \text{Almost Lie}
\end{array}
\end{equation}
where all inclusions are, again, strict, and, moreover,
\begin{equation}\label{eq-eq}
\begin{gathered}
\text{\rm Malcev} \>\cap\> \text{\rm CD} = 
\text{\rm Sagle} \>\cap\> \text{\rm CD} = 
\text{\rm Malcev} \>\cap\> \text{\rm Sagle}
\\
\text{\rm Binary Lie} \>\cap\> \text{\rm Almost Lie} = \text{\rm CD}
\end{gathered}
\end{equation}
(the graphically inclined reader may wish to draw Venn diagrams representing all
this).

Here the variety of ``Almost Lie'' (anticommutative) algebras is defined by
the identity\footnote{
There are other meanings of ``almost Lie'' one can encounter in the literature,
but since neither of them seems to be widespread and accepted, and since almost
Lie algebras in our sense are already mentioned in a number of related papers --
for example \cite{kk} and \cite{kk2} -- we choose to keep this terminology. 
}
\begin{equation}\label{eq-alie}
J(x,y,z)w = 0 .
\end{equation}

\begin{lemma}\label{lemma-1}
The variety {\rm Malcev} $\>\cap\>$ {\rm Sagle} coincides with the variety of 
almost Lie algebras which additionally satisfy the identity
\begin{equation}\label{eq-4}
J(x,y,zw) = 0 .
\end{equation}
\end{lemma}

\begin{proof}
As proved in 
\cite[Lemma 2.10]{sagle-malcev}, any Malcev algebra satisfies the identity
\begin{equation}\label{eq-msagle}
3J(y,z,wx) = J(x,y,z)w - J(y,z,w)x - 2J(z,w,x)y + 2J(w,x,y)z .
\end{equation}

Expressing via this identity all the summands on the right-hand side of 
(\ref{eq-sagle}), which are of the form $J(\tb,\tb,\tb\tb)$, through terms of 
the form $J(\tb,\tb,\tb)\tb$, we get the identity (\ref{eq-alie}), i.e., any 
algebra which is both Malcev and Sagle, is almost Lie. Then (\ref{eq-msagle}) 
implies that any algebra which is both Malcev and Sagle, satisfies also the
identity (\ref{eq-4}).

Conversely, any algebra satisfying both (\ref{eq-alie}) and (\ref{eq-4}) is,
obviously, both Malcev and Sagle.
\end{proof}

\begin{lemma}\label{lemma-bacd}
{\rm Binary Lie} $\>\cap\>$ {\rm Almost Lie} $\subseteq$ {\rm CD}.
\end{lemma}

\begin{proof}
As proved in 
\cite[\S 3]{sagle-binlie}, any binary Lie algebra satisfies the identity
$$
3(J(wx,y,z) + J(yz,w,x)) = - J(x,y,z)w + J(y,z,w)x - J(z,w,x)y + J(w,x,y)z .
$$

If the algebra is simultaneously almost Lie, the right-hand side of this 
identity vanishes, and we are left with the identity
\begin{equation}\label{eq-jj}
J(wx,y,z) + J(yz,w,x) = 0 .
\end{equation}

Using the last identity and anticommutativity, we get
\begin{multline}\label{eq-44}
((xz)w)y + ((yw)z)x \\= 
  J(xz,w,y) + ((xz)y)w - (wy)(xz) 
+ J(yw,z,x) + ((yw)x)z - (zx)(yw) \\=
((xz)y)w + ((yw)x)z
\end{multline}
for any elements $x,y,z,w$ of an algebra which is simultaneously binary Lie and
almost Lie.

Now transform the left-hand side of (\ref{eq-6}):
\begin{multline*}
((xy)a)b - ((xy)b)a - ((xa)b)y + ((xb)a)y + ((ya)b)x - ((yb)a)x 
\\=
  ((xy)a)b - ((xy)b)a
- ((xa)y)b - ((yb)x)a
+ ((xb)y)a + ((ya)x)b
\\= 
  J(x,y,a)b
- J(x,y,b)a = 0 ,
\end{multline*}
where the first equality is obtained by applying the identity (\ref{eq-44}) 
twice, to the pairs formed by the 3rd and 6th summands, and by the 4th and 5th
summands.
\end{proof}

\begin{proposition}\label{prop-1}
All inclusions in the diagram (\ref{eq-inc}) do indeed take place.
\end{proposition}

\begin{proof}
``Lie $\subset$ Malcev $\cap$ Sagle'': Obvious. To see that the inclusion is
strict (we do not need this in what follows, but doing this for completeness), 
consider the free anticommutative algebra freely generated by elements $x,y,z$,
such that any product of any $4$ elements vanishes. This $9$-dimensional 
nilpotent algebra satisfies any identity of degree $4$, in particular, it is 
both Malcev and Sagle, but, obviously, not Lie, as $J(x,y,z) \ne 0$.

\smallskip

``Malcev $\cap$ Sagle $\subset$ CD'': By Lemma \ref{lemma-1}, 
$$
\text{\rm Malcev} \>\cap\> \text{\rm Sagle} \subseteq 
\text{\rm Binary Lie} \>\cap\> \text{\rm Almost Lie} ,
$$
and then apply Lemma \ref{lemma-bacd}.

To show that the inclusion is strict, on can take, for example, the 
one-parametric family of nilpotent algebras ${\bf B}_{6,1}^\alpha$ from \cite{cd-nilp}
(see Theorems 3 and 10 there). These are $6$-dimensional algebras with the basis
$\{e_i\}_{i=1,\dots,6}$ and multiplication table
$$
e_1 e_2 = e_4, \quad e_1 e_3 = e_5, \quad e_2 e_3 = \alpha e_6, \quad 
e_4 e_5 = e_6 ,
$$
where $\alpha \in K$; these algebras are CD, but not Malcev. 

\smallskip

``CD $\subset$ Binary Lie'': As noted in \cite{cd-nilp}, substituting $a=x$ and 
$b=y$ in (\ref{eq-6}) yields (\ref{eq-binlie}). It is not difficult to find 
examples of binary Lie algebras which are not CD; for nilpotent such algebras, see, again, \cite{cd-nilp}. Another nice example is a $7$-dimensional simple 
Malcev algebra (what follows from Proposition \ref{prop-simple} below).

\smallskip

``CD $\subset$ Almost Lie'': Let $A$ be a CD algebra. Write the Jacobi identity
for elements of $\gl(A)$:
\begin{equation}\label{jacobi}
[[R_x,R_y],R_z] + [[R_z,R_x],R_y] + [[R_y,R_z],R_x] = 0
\end{equation}
for any $a,b,c\in A$. 

Write the identity (\ref{eq-rda}) for the derivation $[R_x,R_y]$:
$$
[[R_x,R_y],R_z] = R_{[R_x,R_y](z)}
$$
for any $x,y,z \in A$. Taking the last identity into account, the identity 
(\ref{jacobi}) can be rewritten as
$$
[R_x,R_y](z) + [R_z,R_x](y) + [R_y,R_z](x) \in Z(A).
$$

Due to anticommutativity of $A$ and the fact that the characteristic of the 
ground field is different from $2$, the left-hand side in the last inclusion is
nothing but the Jacobiator $J(x,y,z)$, whence the statement.

Examples of almost Lie algebras which are not CD can be constructed by 
considering central extensions of Lie algebras in the suitable variety; this
is deferred to \S \ref{sec-ext}\footnote{
An alternative method would be to use Albert \cite{albert} to construct 
explicitly the multiplication table of a suitable finite-dimensional homomorphic
image of a free almost Lie algebra, and then verify in some other 
general-purpose computer algebra system like GAP, that this homomorphic image is
not CD. A similar procedure -- for another set of identities -- is described more thoroughly in \cite[\S 4]{mock-lie}.
}.
\end{proof}

\begin{proposition}
All equalities in (\ref{eq-eq}) do indeed take place.
\end{proposition}

\begin{proof}
``Binary Lie $\>\cap\>$ Almost Lie = CD'':
The inclusion ``CD $\subseteq$ Binary Lie $\>\cap\>$ Almost Lie'' follows from 
(\ref{eq-inc}), and the inverse inclusion is proved in Lemma \ref{lemma-bacd}.

\smallskip

``Malcev $\>\cap\>$ CD = Malcev $\>\cap\>$ Sagle'': By Lemma \ref{lemma-1}, the
variety Malcev $\>\cap\>$ Sagle coincides with the variety of almost Lie algebras which, additionally, satisfy the identity
(\ref{eq-4}).

On the other hand, by just proved, 
$$
\text{\rm Malcev} \>\cap\> \text{\rm CD} = 
\text{\rm Malcev} \>\cap\> \text{\rm Binary Lie} \>\cap\> \text{\rm Almost Lie}
= \text{\rm Malcev} \>\cap\> \text{\rm Almost Lie}.
$$

But by (\ref{eq-msagle}), any Malcev algebra which is simultaneously almost 
Lie, satisfies (\ref{eq-4}), what shows that
$$
\text{\rm Malcev} \>\cap\> \text{\rm Almost Lie} = 
\text{\rm Malcev} \>\cap\> \text{\rm Sagle} .
$$

\smallskip

``Sagle $\>\cap\>$ CD = Malcev $\>\cap\>$ Sagle'': The inclusion
Malcev $\>\cap\>$ Sagle $\subseteq$ Sagle $\>\cap\>$ CD follows from 
(\ref{eq-inc}), so let us prove the inverse inclusion.

By the already proved, we have
$$
\text{\rm Sagle} \>\cap\> \text{\rm CD} = 
\text{\rm Sagle} \>\cap\> \text{\rm Binary Lie} \>\cap\> \text{\rm Almost Lie} .
$$

The algebra which is simultaneously binary Lie and almost Lie, satisfies 
(\ref{eq-jj}), and by (\ref{eq-sagle}), the algebra which is simultaneously 
Sagle and almost Lie, satisfies the identity
$$
J(w,z,xy) + J(w,y,zx) + J(w,x,yz) = 0 .
$$

Permuting in the last identity $z$ and $w$, and using (\ref{eq-jj}), we get
$$
- J(w,z,xy) + J(w,x,yz) + J(w,y,zx) = 0 .
$$

The last two identities yield the identity (\ref{eq-4}), and thus, by 
Lemma \ref{lemma-1}, 
$$
\text{\rm Sagle} \>\cap\> \text{\rm Binary Lie} \>\cap\> \text{\rm Almost Lie}
\subseteq \text{\rm Malcev} \>\cap\> \text{\rm Sagle} . 
$$
\end{proof}

Perhaps, the most important among all these multiple relations is the inclusion
\begin{equation}\label{eq-inc1}
\text{CD} \subset \text{Almost Lie}
\end{equation}
which shows that, after all, CD algebras are not that far from the Lie ones. 

An immediate corollary of this inclusion is

\begin{proposition}\label{prop-simple}
For any CD algebra $A$, the quotient $A/Z(A)$ is a Lie algebra. In particular, 
any centerless (and, in particular, simple) CD algebra is a Lie algebra.
\end{proposition}

Thus, CD algebras are, essentially, central extensions, in the suitable variety,
of Lie algebras. As central extensions should be described by second degree
cohomology, this suggests that there should be a ``CD cohomology'', extending
the usual Chevalley--Eilenberg cohomology, responsible for such central 
extensions. And indeed, such cohomology is constructed in \S \ref{sec-coho}.

It is natural to ask whether any simple binary Lie algebra is Malcev (and hence
is either Lie, or is a $7$-dimensional simple Malcev algebra). For 
finite-dimensional algebras over a field of characteristic zero the question was
answered in affirmative in \cite{grishkov}, while the cases of positive 
characteristic, and of infinite-dimensional algebras remain open (see, e.g., 
\cite[Problems 2.33 and 3.87]{dniester} and \cite[p.~263]{arenas-shest}). 
Proposition \ref{prop-simple} shows that in the narrower class of CD algebras
the answer is also affirmative.

Note also that Proposition \ref{prop-simple} implies that CD non-Lie algebras of
dimension $\le 5$ listed in \cite{cd-nilp1} actually exhaust \emph{all} CD 
non-Lie algebras in those dimensions.

\section{Further relations with Lie algebras}\label{sec-lie}

Let $A$ be a CD algebra. Consider the subspace $R(A)$, spanned by all maps 
$R_a$, $a \in A$, and the subalgebra $\Der(A)$ in the Lie algebra $\gl(A)$ of
all linear maps on $A$ (subject to the usual commutator of linear maps), and 
consider their formal direct sum $R(A) \oplus \Der(A)$ (``formal'', as they may
intersect, so the sum $R(A) + \Der(A)$, considered as the linear subspace of 
$\gl(A)$, is not necessarily direct; more on that below). As $A$ is a 
CD algebra, we have $[R(A),R(A)] \subseteq \Der(A)$. Moreover, because of 
(\ref{eq-rda}), it holds $[\Der(A),R(A)] \subseteq R(A)$. Thus 
$R(A) \oplus \Der(A)$ is a Lie algebra with respect to the usual commutator, 
actually a semidirect sum with $\Der(A)$ acting on $R(A)$.

This construction is completely analogous to those in the commutative case 
(cf., e.g., \cite[Definition II.1.4]{bertram}), and similar to the construction
of the structure algebra of a Jordan algebra used in the Kantor--Koecher--Tits 
construction (cf., e.g., \cite[Chapter VIII, \S 4]{J}); the construction of the 
holomorph of a Lie algebra (i.e., the semidirect sum $L \oplus \Der(L)$ for a 
Lie algebra $L$) is somewhat similar in spirit, but different, as in the 
holomorph we have $[L,L] \subseteq L$.

This construction can be modified in several ways. For example, instead of the
linear space $R(A)$ we may consider the Lie multiplication algebra $M(A)$ (again
subject to the commutator). As $M(A)$ is generated by $R(A)$, the commutation
relations in the Lie algebra $R(A) \oplus \Der(A)$ can serve as defining 
relations in the Lie algebra $M(A) \oplus \Der(A)$.

Another possibility is to consider not formal, but ``real'' direct sum, i.e.,
the Lie subalgebra of $\gl(A)$, spanned by $R(A)$ and $\Der(A)$. To consider 
this variant more thoroughly, define the \emph{Lie center} of $A$, denoted by 
$LZ(A)$, as the set of all elements $z \in A$ such that $J(a,b,z) = 0$ for any 
$a,b \in A$. Obviously, the Lie center is always a vector subspace of $A$, 
and $LZ(A) = A$ if and only if $A$ is a Lie algebra. In a sense, it serves as a
measure of ``non-Lieness'' of an algebra.

In what follows, it will be convenient to identify elements of $A$ with the 
corresponding multiplications in $R(A)$, up to the center. Namely, the kernel of
the linear map $A \to R(A)$, $a \mapsto R_a$ coincides with $Z(A)$; hence we 
have the isomorphism of vector spaces 
\begin{equation}\label{eq-i}
A/Z(A) \xrightarrow{\sim} R(A) .
\end{equation}
Due to Proposition \ref{prop-simple}, this is also an isomorphism of Lie 
algebras.

\begin{lemma}\label{lemma-jz}
For any CD algebra $A$, there is isomorphism of Lie algebras
$$
LZ(A)/Z(A) \simeq R(A) \cap \Der(A) .
$$
\end{lemma}

\begin{proof}
It is obvious that $Z(A) \subseteq LZ(A)$. The condition $z \in LZ(A)$ is 
equivalent, taking into account the anticommutativity of $A$, to the condition 
$R_z \in \Der(A)$. Hence the image of $LZ(A)$ under the isomorphism (\ref{eq-i})
coincides with $R(A) \cap \Der(A)$.
\end{proof}

\begin{lemma}\label{lemma-2}
For any CD algebra $A$, $LZ(A)$ is an ideal of $A$.
\end{lemma}

\begin{proof}
Let $z \in LZ(A)$, and $x \in A$. By Lemma \ref{lemma-jz}, we have 
$R_z \in \Der(A)$, and then by (\ref{eq-rda}) we have 
$[R_z,R_x] = R_{R_z(x)} = R_{xz}$. Then, since $A$ is a CD algebra, 
$R_{xz} \in \Der(A)$, and hence, again by Lemma \ref{lemma-jz}, $xz \in LZ(A)$.
\end{proof}

Thus, identifying $R(A)$ with $A/Z(A)$ via isomorphism (\ref{eq-i}), the 
intersection between $R(A)$ and $\Der(A)$ is identified with $LZ(A)/Z(A)$ by
Lemma \ref{lemma-2}, and (generally, non-direct) sum $R(A) + \Der(A)$ can be
identified with the direct sum $A/LZ(A) \oplus \Der(A)$, the Lie subalgebra of
$\gl(A)$.

\section{``Naive'' cohomology}\label{sec-coho}

The goal of this section is to construct cohomology theory of CD algebras. The 
standard approach to construct cohomology in a variety of algebras is an 
operadic one: provided that the corresponding operad $\mathcal P$ is quadratic 
and Koszul, there is a small explicit cochain complex built out of the Koszul 
dual cooperad $\mathcal P^{\mbox{!`}}$. However, this works only for quadratic 
Koszul operads. The operad defined by the identity (\ref{eq-6}) is cubic. The 
standard trick, employed, for example, in the case of the Jordan operad, is to 
pass to triple systems; the corresponding category of triple systems should be 
equivalent (say, in representation-theoretic sense) to the initial category of 
binary algebras, and the corresponding operad will be (ternary) quadratic. 
However, it is not immediately clear which triple systems should correspond to 
CD algebras, and whether the corresponding operad will be Koszul (we believe it
will be not).

Thus we rely on the ``naive'' approach to cohomology. Under this, we mean an 
attempt to construct (the beginning of) the corresponding cochain complex by 
utilizing the low-degree structural interpretations of cohomology in the given 
variety: derivations, central extension, deformations, etc. We take cohomology 
of Lie algebras as a model.

For that, we need first to define what is a module over a CD algebra is. We 
follow a nowadays standard approach which goes back to Eilenberg (see, for 
example, \cite[Chapter II, \S 5]{J}). Namely, for an algebra $A$ in a given 
variety, a vector space $M$ with an $A$-action on it, is declared a module over
$A$, if the semidirect sum $A \oplus M$, where multiplication between elements 
of $A$ and $M$ is determined by the given action, and multiplication on $M$ is 
zero, belongs to the same variety. According to this approach, a vector space $M$ with a (left) action of
a CD algebra $A$, denoted by $a \bullet m$, is called a \emph{module over $A$},
if the following equality holds:
$$
(xy)a \bullet m + a \bullet ((xy) \bullet m)
- x \bullet ((ya) \bullet m) + y \bullet ((xa) \bullet m)
- x \bullet (a \bullet (y \bullet m)) + y \bullet (a \bullet (x \bullet m)) = 0
.
$$
for any $x,y,a \in A$ and $m \in V$.

As in the Lie case, $A$, considered as a (left) module over itself, is called 
the \emph{adjoint module}.

As we are interested in central extensions and deformations of CD algebras, 
we start with the second cohomology. The second cohomology is interpreted as 
equivalence classes of square-zero extensions. Namely, let $A$ be a CD algebra 
and $M$ an $A$-module, and consider the CD algebra structure on the direct sum 
of vector space $A \oplus M$, where multiplication on $A$ is given by the 
formula $x * y = xy + \varphi(x,y)$ for some bilinear map 
$\varphi: A \times A \to M$, multiplication between $A$ and $M$ is given by action of $A$ on $M$, and multiplication on $M$ is zero. Then
$\varphi$ is skew-symmetric, and
\begin{multline}\label{eq-2}
  \varphi((xy)a,b) - \varphi((xy)b,a) 
- \varphi((xa)b,y) + \varphi((xb)a,y) + \varphi((ya)b,x) - \varphi((yb)a,x)
\\
+ a \bullet \varphi(xy,b) - b \bullet \varphi(xy,a) 
- x \bullet \varphi(ya,b) + x \bullet \varphi(yb,a) 
+ y \bullet \varphi(xa,b) - y \bullet \varphi(xb,a) 
\\
- a \bullet (b \bullet \varphi(x,y)) + b \bullet (a \bullet \varphi(x,y))
- x \bullet (a \bullet \varphi(y,b)) + x \bullet (b \bullet \varphi(y,a))
- y \bullet (b \bullet \varphi(x,a)) + y \bullet (a \bullet \varphi(x,b)) 
\\ = 0
\end{multline}
for any $x,y,a,b \in A$.

The usual notion of equivalence of square-zero extensions 
$0 \to M \to \cdot \to A \to 0$ leads to the notion of trivial, or split, 
extension, which corresponds to a cocycle of the form 
\begin{equation}\label{eq-cob2}
\varphi(x,y) = \psi(xy) - x \bullet \psi(y) + y \bullet \psi(x) 
\end{equation}
for any $x,y \in A$ and some linear map $\psi: A \to M$. Thus the right-hand
side of this equality suggests the definition of the first order cocycles, what
confirms the standard interpretation of the first cohomology as outer 
derivations. 

The inner derivations of $A$, according to the general approach devised by 
Schafer (see, for example, \cite[Chapter II, \S 3]{schafer}), are defined as 
derivations lying in the Lie multiplication algebra $M(A)$. Since 
$[R_a,R_b] \in \Der(A)$ and due to (\ref{eq-rda}), $M(A)$ is linearly spanned by
linear maps of the form $R_a$ and $[R_a,R_b]$ for $a,b \in A$. On the other 
hand, by Lemma \ref{lemma-jz}, $R_a$ is a derivation of $A$ if and only if 
$a \in LZ(A)$. Thus any inner derivation of $A$ is of the form 
$R_z + \sum_i [R_{a_i},R_{b_i}]$ for some $z \in LZ(A)$ and $a_i,b_i \in A$.

Another interpretation of the second cohomology with coefficients in the adjoint
module is equivalence classes of infinitesimal deformations of an algebra. Thus,
following nowadays standard approach by Gerstenhaber, for a CD algebra $A$ 
consider a deformed algebra over the ring $K[[t]]$, with multiplication
$$
x*y = xy + \varphi_1 (x,y) t + \varphi_2(x,y) t^2 + \dots
$$
That the deformed algebra is also CD algebra, is equivalent to an (infinite) 
series of equalities, obtained by collecting coefficients by powers of $t$ in
the CD identity (\ref{eq-6}) for the multiplication $*$. The zeroth of these 
equalities (coefficients by $t^0$) coincides with the CD identity for the 
original multiplication in $A$, and thus gives nothing new. The first of these 
equalities (coefficients by $t^1$) is obtained by ignoring all terms with powers
of $t$ higher than $1$, and thus is equivalent to the CD identity in the algebra
$A \oplus At$, where $At$ is the adjoint module over $A$ with trivial 
multiplication. Hence it is equivalent to the identity (\ref{eq-2}) with $M=A$ 
and $\bullet$ being multiplication in the algebra. The further equalities 
(coefficients by $t^2$ and higher degrees) involve Massey brackets of 
$\varphi_i$'s which can be interpreted as obstructions of prolongations of an 
infinitesimal deformation to a global one, and lie in the third cohomology of 
$A$ with coefficients in the adjoint module. But as it leads to cumbersome 
formulae, and our primary interest is in the second cohomology, we will not 
pursue this further.

The condition of triviality of such deformation, i.e., the equivalence to the 
initial algebra with respect to a homomorphism of the form
$$
\psi(x) = x + \psi_1(x) t + \psi_2(x) t^2 + \dots
$$
leads to an infinite series of equalities, the first of which is
$$
\varphi_1(x,y) = \psi_1(xy) - \psi_1(x)y - x\psi_1(y) ,
$$
what is the partial case of (\ref{eq-cob2}) in the case of the adjoint module.

Putting all this together, we define the initial terms of the cochain complex
associated to a CD algebra $A$ and an $A$-module $M$:
\begin{equation}\label{eq-c}
0 \to LZ(M) \oplus (A \otimes M) \overset{\dcobound^0}\longrightarrow \C^1(A,M) 
\overset{\dcobound^1}\longrightarrow \C^2(A,M) 
\overset{\dcobound^2}\longrightarrow \C^4(A,M) .
\end{equation}
Here $LZ(M)$ is a ``Lie center'' of $M$ defined as
$$
LZ(M) = \set{m \in M}
{xy \bullet m - x \bullet (y \bullet m) + y \bullet (x \bullet m) = 0 
\text{ for any } x,y \in M} ,
$$ 
and $\C^n(A,M)$ for $n \ge 1$ is a linear space of skew-symmetric linear maps 
$\underbrace{A \times \dots \times A}_n \to M$. The differentials are defined as
follows:
$$
\dcobound^0(m)(b) = b \bullet m 
$$
for $b \in A$, $m \in LZ(M)$,
$$
\dcobound^0(a \otimes m)(b) = a \bullet (b \bullet m) - b \bullet (a \bullet m)
$$
for $a,b \in A$, $m \in M$. The ``$LZ(M)$'' component of $\dcobound^0$ is 
similar to the zeroth differential in the Chevalley--Eilenberg complex 
computing the Lie algebra cohomology, and the ``$A \otimes M$'' component is 
similar to the zeroth differential in the complex computing cohomology of 
quadratic Jordan algebras defined in \cite[\S I.3]{mccrimmon}.

Further,
$$
\dcobound^1(\varphi)(x,y) = 
\varphi(xy) - x \bullet \varphi(y) + y\bullet \varphi(x)
$$
for $\varphi \in \C^1(A,M)$, $x,y \in A$, and
\begin{multline}\label{eq-dd}
\dcobound^2(\varphi)(x,y,a,b) \\ =
  \varphi((xy)a,b) - \varphi((xy)b,a) 
- \varphi((xa)b,y) + \varphi((xb)a,y) + \varphi((ya)b,x) - \varphi((yb)a,x)
\\
+ a \bullet \varphi(xy,b) - b \bullet \varphi(xy,a) 
- x \bullet \varphi(ya,b) + x \bullet \varphi(yb,a) 
+ y \bullet \varphi(xa,b) - y \bullet \varphi(xb,a) 
\\ \hskip -2pt
- a \bullet (b \bullet \varphi(x,y)) + b \bullet (a \bullet \varphi(x,y))
- x \bullet (a \bullet \varphi(y,b)) + x \bullet (b \bullet \varphi(y,a))
- y \bullet (b \bullet \varphi(x,a)) + y \bullet (a \bullet \varphi(x,b))
\end{multline}
for $\varphi \in \C^2(A,M)$, $x,y,a,b \in A$.

The equalities $\dcobound^1 \circ \dcobound^0 = 0$ and
$\dcobound^2 \circ \dcobound^1 = 0$ are verified in a straightforward, if not a
bit cumbersome way, or just follow from the structural interpretations described
above. 

This is not the only sensible way to define cohomology of CD algebras. For 
example, one may argue that a proper notion of derivation in this context is the
following: a linear map $D:A \to A$ such that the semidirect sum $A \oplus KD$,
where multiplication between $A$ and $D$ is determined by action of $D$ on $A$.
In the variety of Lie algebras, this leads to the usual notion of derivation, 
but in the variety of CD algebras, this leads to what might be called a 
\emph{CD derivation} of a CD algebra $A$: a linear map $D:A \to A$ such that
$$
D((xy)a) - D(xy)a - D(xa)y + D(ya)x + (D(x)a)y - (D(y)a)x = 0
$$
for any $x,y,a \in A$. An inner derivation in this context is, as in the variety
of Lie algebras, just a multiplication $R_a$, $a\in A$. Accordingly, one may 
define the initial terms of the cochain complex responsible for cohomology of a
CD algebra $A$ with coefficients in an $A$-module $M$, as
$$
0 \to M \overset{\dcobound^0}\longrightarrow \C^1(A,M) 
        \overset{\dcobound^1}\longrightarrow \C^3(A,M) ,
$$
where
$$
\dcobound^0(m)(x) = x \bullet m
$$
for $m\in M$ and $x \in A$, and
\begin{equation}\label{eq-d1}
\dcobound^1(\varphi)(x,y,a) = 
  \varphi((xy)a) - a \bullet \varphi(xy) 
- y \bullet \varphi(xa) + x \bullet \varphi(ya) 
+ y \bullet (a \bullet \varphi(x)) - x \bullet (a \bullet \varphi(y))
\end{equation}
for $\varphi \in \C^1(A,M)$ and $x,y,a \in A$.

One cannot say which variant of these cochain complexes is ``better''. In that 
respect, the situation resembles those with cohomology of so-called mock-Lie 
algebras (commutative algebras satisfying the Jacobi identity), where also one 
cannot define a low-degree cohomology in a canonical and coherent way, basing on
structural interpretations (cf. \cite[\S 1]{mock-lie}). However, as we are 
interested in central extensions and deformations, we adopt the first variant of
the cochain complex, and define the second degree cohomology of a CD algebra $A$
with coefficients in an $A$-module $M$ as 
$\Homol_{CD}^2(A,M) = \Ker \dcobound^2 / \im \dcobound^1$, where $\dcobound^1$
and $\dcobound^2$ are as in (\ref{eq-c}).

One may try to generalize the formulae for differentials above as follows. Let
$n > 0$, and \linebreak $\dcobound: \C^n(A,M) \to \C^{n+2}(A,M)$ is given by 
\begin{multline*}
\dcobound(\varphi)(x,y,a_1, \dots, a_n) 
\\ = 
\sum_{i=1}^n (-1)^i \Big(
  \varphi((xy)a_i, a_1, \dots, \widehat{a_i}, \dots, a_n)
+ a_i \bullet \varphi(xy, a_1, \dots, \widehat{a_i}, \dots, a_n)
\\ \hspace{56pt}
- x \bullet \varphi(ya_i, a_1, \dots, \widehat{a_i}, \dots, a_n)
+ y \bullet \varphi(xa_i, a_1, \dots, \widehat{a_i}, \dots, a_n)
\\ \hspace{102pt}
- x \bullet (a_i \bullet \varphi(y, a_1, \dots, \widehat{a_i}, \dots, a_n))
+ y \bullet (a_i \bullet \varphi(x, a_1, \dots, \widehat{a_i}, \dots, a_n))
\Big)
\\ + \sum_{1 \le i < j \le n} (-1)^{i+j+n+1} \Big(
  \varphi(
(xa_i)a_j, y, a_1, \dots, \widehat{a_i}, \dots, \widehat{a_j}, \dots, a_n)
- \varphi(
(xa_j)a_i, y, a_1, \dots, \widehat{a_i}, \dots, \widehat{a_j}, \dots, a_n)
\\ \hspace{104pt}
- \varphi(
(ya_i)a_j, x, a_1, \dots, \widehat{a_i}, \dots, \widehat{a_j}, \dots, a_n)
+ \varphi(
(ya_j)a_i, x, a_1, \dots, \widehat{a_i}, \dots, \widehat{a_j}, \dots, a_n)
\\ \hspace{69pt}
+ a_i \bullet (a_j \bullet 
\varphi(x, y, a_1, \dots, \widehat{a_i}, \dots, \widehat{a_j}, \dots, a_n))
- a_j \bullet (a_i \bullet 
\varphi(x, y, a_1, \dots, \widehat{a_i}, \dots, \widehat{a_j}, \dots, a_n))
\Big)
\end{multline*}
for $\varphi \in \C^n(A,M)$, and $x,y,a_1, \dots, a_n \in A$.

One can prove, in the absence of analogs of Cartan formulas in the 
Chevalley--Eilenberg complex (are there ones?), by direct verification if not 
without some pain, that $\dcobound \circ \dcobound = 0$, so we get, in fact, two
complexes, which lead to what may be called ``odd'' and ``even'' CD cohomology
respectively: 
\begin{gather*}
\C^1(A,M) \overset{\dcobound}{\to} \C^3(A,M) \overset{\dcobound}{\to} 
\C^5(A,M) \overset{\dcobound}{\to} \dots  
\\
\C^2(A,M) \overset{\dcobound}{\to} \C^4(A,M) \overset{\dcobound}{\to} 
\C^6(A,M) \overset{\dcobound}{\to} \dots
\end{gather*}
The differential $\dcobound: \C^1(A,M) \to \C^3(A,M)$ here coincides with 
differential (\ref{eq-d1}), and the differential 
$\dcobound: \C^2(A,M) \to \C^4(A,M)$ coincides with differential (\ref{eq-dd}).

However, we will not pursue this topic further and in the subsequent section
will work exclusively with $\Homol_{CD}^2(A,M)$ as defined above.

\section{Central CD extensions of Lie algebras}\label{sec-ext}

In this section we discuss central CD extensions, or, in other words, 
$\Homol_{CD}^2(L,K)$, for various Lie algebras $L$ with coefficients in the trivial module $K$. According to the definition, the vector
space $\Homol_{CD}^2(L,K)$ is a quotient of \emph{CD $2$-cocycles} by 
$2$-coboundaries:
$$
\Homol_{CD}^2(L,K) = \frac{\Z_{CD}^2(L,K)}{\B^2(L,K)} .
$$

The space $\Z_{CD}^2(L,K)$ of CD $2$-cocycles consists of skew-symmetric bilinear
maps $\varphi: L \times L \to K$ satisfying the condition
\begin{equation}\label{eq-2c}
  \varphi([[x,y],a],b) - \varphi([[x,y],b],a) 
- \varphi([[x,a],b],y) + \varphi([[x,b],a],y) 
+ \varphi([[y,a],b],x) - \varphi([[y,b],a],x) = 0
\end{equation}
for any $x,y,a,b \in L$, and the space of $2$-coboundaries $\B^2(L,K)$ consists,
as in the Lie case, of bilinear maps of the form $\varphi(x,y) = \psi([x,y])$ 
for some linear map $\psi: L \to K$.

Note that for any Lie algebra $L$, the usual Chevalley--Eilenberg cohomology
$\Homol^2(L,K)$ is a subspace in the CD cohomology $\Homol_{CD}^2(L,K)$.

Our first goal is to obtain examples of an almost Lie algebra which is not
CD, promised at the end of the proof of Proposition \ref{prop-1}, by considering
one-dimensional central extensions of a Lie algebra $L$. Such central extensions
can be written as the vector space direct sum $L \oplus Kz$, where 
multiplication in $L$ is twisted by a $2$-cocycle $\varphi: L \times L \to K$,
$\{x,y\} = [x,y] + \varphi(x,y)z$, and $z$ is a central element. Such an algebra
is almost Lie for any skew-symmetric $\varphi$, while it is a CD algebra if and
only if $\varphi$ is CD $2$-cocycle. Thus any Lie algebra $L$ whose second CD 
cohomology $\Homol_{CD}^2(L,K)$ is strictly larger then its second 
Chevalley--Eilenberg cohomology $\Homol^2(L,K)$, will lead, by extending $L$ by
any CD $2$-cocycle which is not a Chevalley--Eilenberg cocycle, to an example of
a CD algebra which is not Lie. Similarly, a Lie algebra $L$, for which 
$\Homol_{CD}^2(L,K)$ is strictly smaller than the space $\C^2(L,K) / \B^2(L,K)$ 
-- which can be considered as the ``2nd almost Lie cohomology'' -- will lead to
an example of an almost Lie algebra which is not CD.

Obviously, abelian Lie algebras do not qualify for such examples, so let us look
at nonabelian Lie algebras of low dimension. Elementary calculations show that
the two-dimensional nonabelian Lie algebra, and the $3$-dimensional nilpotent 
Lie algebra do not qualify either, as for these algebras
$$
\Homol^2(L,K) = \Homol_{CD}^2(L,K) = \frac{\C^2(L,K)}{\B^2(L,K)}
$$
(all these three spaces vanish in the case where $L$ is two-dimensional 
nonabelian, and are of dimension $2$ in the case where $L$ is $3$-dimensional 
nilpotent), but the direct sum of the two-dimensional nonabelian and the one-dimensional algebra 
does qualify, as for this algebra we have
$$
\Homol^2(L,K) = \Homol_{CD}^2(L,K) \subsetneq \frac{\C^2(L,K)}{\B^2(L,K)}
$$
(the corresponding spaces being of dimensions $1$ and $2$).

Further low-dimensional solvable and nilpotent Lie algebras of dimension $3$ and
higher could provide a plethora of such examples (including the cases where 
$\Homol^2(L,K)$ and $\Homol_{CD}^2(L,K)$ do not coincide).

On the other hand, for simple Lie algebras we have

\begin{theorem}[``Second CD Whitehead lemma'']
For any simple finite-dimensional Lie algebra $L$ over a field of characteristic
zero, and any finite-dimensional $L$-module $M$, $\Homol_{CD}^2(L,M) = 0$.
\end{theorem}

\begin{proof}
This follows at once from \cite[Theorem 6]{grishkov}. Indeed, it is proved there
that any solvable (and hence abelian) extension of $L$ in the variety of binary
Lie algebras (and hence in the variety of CD algebras) splits.
\end{proof}

In the positive characteristic, we merely have a 

\begin{conjecture}
For any simple finite-dimensional Lie algebra $L$ over a field of characteristic
$\ne 2,3$, $\Homol_{CD}^2(L,K) = \Homol^2(L,K)$.
\end{conjecture}

The conjecture is supported by computer calculations for algebras of small 
dimension.

\section{Further questions}

1)
How ``far'' a CD algebra can be from Lie algebras? To start with, describe CD 
algebras $A$ with the Lie center ``as small as possible'', i.e., satisfying the
condition $LZ(A) = Z(A)$.

\smallskip

2) 
Which ``interesting'' Lie algebras can be realized as constructions described
in \S \ref{sec-lie}?

\smallskip

3) Study free CD algebras. Are they central extensions of free Lie algebras?

\smallskip

4)
For an algebra $A$, its minus-algebra $A^{(-)}$ is an algebra defined on the
same underlying vector space $A$ subject to multiplication given by the
commutator $[x,y] = xy-yx$. If the minus-algebra of an algebra $A$ belongs to
a (necessary anticommutative) variety $\mathcal V$, then $A$ is called 
$\mathcal V$-admissible. Most of the distinguished varieties of algebras have 
they ``admissible'' counterparts: thus, associative algebras are Lie-admissible,
alternative algebras are Malcev-admissible, and binary Lie algebras are 
assocyclic-admissible, where the variety of assocyclic algebras is defined by 
the identity
$$
(x,y,z) = (z,x,y) ,
$$
$(x,y,z) = (xy)z - x(yz)$ being the associator of elements $x,y,z$. 

In such situation arises the question whether each algebra in an anticommutative
variety $\mathcal V$ is special, i.e., can be embedded into an algebra of the 
form $A^{(-)}$, where $A$ is $\mathcal V$-admissible algebra. Thus, Lie algebras
are special due to the celebrated Poincar\'e--Birkhoff--Witt theorem; the 
speciality of Malcev algebras was a long-standing problem whose negative 
solution was announced recently by Ivan Shestakov; and binary Lie algebras are 
not necessarily special either, see \cite{arenas-shest}.

What would be a natural variety of CD admissible algebras? (One possible 
general approach in operadic language to such sort of questions is described in
\cite{kolesnikov}). Would CD algebras be special with respect to that variety? 

\smallskip

5)
Study representations of CD algebras. Does an analog of the Ado theorem hold,
i.e., whether each finite-dimensional CD algebra admits a faithful 
finite-dimensional representation?

\smallskip

6) 
The classical Lie theory establishes correspondence between Lie groups and Lie 
algebras. This correspondence has been generalized to Malcev algebras, Bol 
algebras, Lie triple systems, Sabinin algebras, etc. (see, for example, 
\cite{gpi} and references therein). In these ``generalized Lie'' correspondences
Lie groups are replaced by various kinds of analytic loops; thus, binary Lie 
algebras correspond to diassociative loops (i.e., the subloop generated by any 
two elements is a group). Which loops would correspond to CD algebras? One of 
the possible approaches to this question would be to find the class of loops 
corresponding to almost Lie algebras, and then, according to (\ref{eq-eq}), take
the intersection of that class with the class of diassociative loops.

\smallskip

7)
Let us drop the commutativity and anticommutativity conditions altogether, and
consider the variety of algebras defined just by the properties that the 
commutator of any two left or right multiplications is a derivation (i.e.,
taking all possible combinations of left/right multiplications, we get $3$ 
defining identities, see, for example, \cite[\S 1]{kk}). Is this variety of 
algebras amenable to study? Low-dimensional nilpotent algebras in this variety 
were classified in \cite{kk} and \cite{kk2}.

\section*{Acknowledgements}

Thanks are due to Askar Dzhumadil'daev, Alexander Grishkov, and 
Jos\'e Mar\'ia P\'erez-Izquierdo for useful discussions. GAP \cite{gap} and 
Albert~\cite{albert} were used to check some of the computations performed in 
this paper. This work was started when Zusmanovich was visiting Federal University of ABC, supported by FAPESP, grant number 18/00726-5.

\renewcommand{\refname}{Software}

\end{document}